\documentclass[a4paper, 12pt]{article}%
\usepackage{amsmath, amsthm, amssymb}
\newtheorem{proposition}{Proposition}[section]
\newtheorem{lemma}{Lemma}[section]
\newtheorem{remark}{Remark}[section]
\newtheorem{corollary}{Corollary}[section]
\usepackage{amsmath}
\usepackage{amsfonts}
\usepackage{amssymb}
\usepackage{amsthm}
\usepackage{graphicx}%
\usepackage{plain}
\usepackage{amsthm}
\providecommand{\U}[1]{\protect\rule{.1in}{.1in}}
\def\a{\alpha}
\def\b{\beta}
\def\sig{\sigma}
\def\Sig{\Sigma}
\def\ga{\gamma}
\def\Ga{\Gamma}
\def\la{\lambda}

\def\th{\theta}
\def\bth{\boldsymbol \theta}
\def\del{\delta}

\def\vareps{\varepsilon}

\def\ov{\overline}
\def\ova{\overline{\a}}
\def\ovb{\overline{\b}}
\def\ovth{\overline{\th}}

\def\dover{\buildrel d \over=}
\def\wtilde{\widetilde}
\def\inar1{INAR\,($1$)}
\def\ar1{AR\,($1$)}
\def\bzp{$\mathbb{Z}_+$}

\def\iid{i.i.d.\ }

\begin{document}

\begin{center}
{\large\bf On the novel geometric and negative binomial \inar1 process}

\vskip 1 cm

Nadjib Bouzar\\
Department of Mathematical Sciences,\\
University of Indianapolis,\\
Indianapolis, IN 46227, USA\\
Email: nbouzar@uindy.edu\\
ORCID iD: 0000-0002-4841-5841

\vskip 1.5 cm

{\large \bf Abstract}
\end{center}
Guerrero et al. \cite{GBSO} propose a novel approach to building first-order integer-valued autoregressive (\inar1) models based on the concept of thinning. The standard approach requires that the thinning operator be defined first and \inar1 models with either a specified marginal (the forward approach) or a specified innovation (the backward approach) are developed. In contrast, the approach in \cite{GBSO} is to start out by specifying both the marginal distribution of the process and that of its innovation sequence, and then proceed to identify the thinning operator by solving a functional equation. In this article we discuss the connection between the thinning operators the authors obtained for their novel geometric and negative binomial \inar1 models and the thinning operator introduced in \cite{AB1} and \cite{AB2}. More specifically, we show that the existence of the two models has been established in \cite{AB1} using the forward approach and a different parameterization. In the process, we strenghthen some of the authors' results obtained for the novel geometric \inar1 process and we extend their results to the novel negative binomial \inar1 process.
 

\newpage

\section{Introduction}

Guerrero et al. \cite{GBSO} propose a novel approach to building stationary first-order integer-valued autoregressive (\inar1) processes. based on the concept of thinning. The standard approach requires that the thinning operator be defined first and \inar1 models with either a specified marginal (the forward approach) or a specified innovation (the backward approach) are developed. In contrast, the approach in \cite{GBSO} is to start out by specifying both the marginal distribution of the process and that of its innovation sequence, and then proceed to identify the thinning operator by solving a functional equation. In this article we discuss the connection between the thinning operators the authors derived for their novel geometric and negative binomial \inar1 models and the thinning operator introduced in \cite{AB1} and \cite{AB2}. More specifically, we show that the existence of the two models has been established in \cite{AB1} using the forward approach that only requires the marginal distribution of the process be negative binomial. It is shown in particular that several distributional properties in \cite{GBSO} follow essentially from results established in \cite{AB1} under a different parameterization. In the process we strenghthen some of the authors' results obtained for the novel geometric \inar1 process by simplifying some of their assumptions and proofs and we proceed to extend their results to the novel negative binomial \inar1 process. We also  discuss in depth estimation procedures of key parameters of the novel negative binomial \inar1 process by capitalizing on results established by several authors for branching processes with immigration.

The paper is organized as follows. We recall some basic facts in  Section 2. Section 3 is devoted to the novel geometric \inar1 process introduced in \cite{GBSO}. We simplify a few results, along with their proofs, and remove some unneeded assumptions for a couple of key results. The authors' novel negative binomial \inar1 process is studied in detail in Section 4. It is shown that its distributional properties follow from results established in \cite{AB1}. Estimation procedures of its key parameters are the object of Section 5.

We will use the notation $\ov{a}=1-a$ for any $a \in [0,1]$ throughout the article. We also note that when we cite the reference \cite{GBSO} we are also including the supplementary material accessible from the journal's website. Finally, we refer to the shifted geometric($p$) distribution, $0<p<1$, as the distribution with pmf $\{f_k\}$ and pgf $P(s)$
\begin{equation}
f_k=\overline{p}p^{k-1}  (k \ge 1)  \quad \hbox{and} \quad P(s)=\frac{\ov{p}s}{1-ps}. \label{shifted_Geo}
\end{equation}    

\section{Some basic facts}

We start out by recalling some basic facts. Let $\bth$ be a possibly multivariate parameter and $X$ a \bzp-valued random variable. The thinning operator $\bth\star X$  is defined by the equation
\begin{equation}
\bth\star X=\sum_{i=1}^X G_i,\label{thin_star}
\end{equation} 
where $(G_i, i\ge 1)$ is a sequence of iid \bzp-valued random variables independent of $X$. The probability generating function (pgf) $\Psi_{\bth\star X}(s)$ of $\bth\star X$ satisfies the equation
\begin{equation}
\Psi_{\bth\star X}(s)=\Psi_X(\Psi_G(s)),\label{pgf_thin_star}
\end{equation}
where $\Psi_X(s)$ is the pgf of $X$ and $\Psi_G(s)$ is the marginal pgf of the sequence $\{G_i\}$. 
 
The operation $\star$ is to be thought of as a scalar multiplication that preserves the discrete nature of the factor $X$. The moment properties of $\bth\ast X$ are deduced from those of $X$ from standard random summation techniques (see \cite{GBSO} and the accompanying supplementary material).

A sequence $(X_t, t\ge 0)$ of \bzp-valued random variables is an \inar1 process if for any $t \ge 0$,
\begin{equation}
X_{t+1}=\bth\star X_t + \vareps_{t+1},\label{eq_inar1}
\end{equation}
where $(\vareps_t, t\ge 0)$ is an \iid sequence of \bzp-valued random variables, referred to as the innovation sequence of $\{X_t\}$. The sequences $\{\vareps_t\}$  and $\{G_i\}$ (of (\ref{thin_star})) are assumed independent.

It is clear that the \inar1 process $\{X_t\}$ of (\ref{eq_inar1}) is a branching process with stationary immigration, and thus is a Markov chain (cf. \cite{ATNE}).  As such, necessary and sufficient conditions for the stationarity of $\{X_t\}$ are well  known. They are based essentially on the existence of a stationary distribution $(\pi_k, k\ge 0)$ for $\{X_t\}$. In the forward approach, discrete self-decomposability of the marginal distribution insures the existence of $\{\pi_k\}$ (see Proposition 2.5 and Lemma 4.1 in \cite{AB1}). In the backward approach, an integrability condition or, equivalently, a log-moment condition on $\vareps_t$ insures the existence of $\{\pi_k\}$ (see Proposition 2.1 in \cite{AB1} and references therein). Either one of the two conditions  can be used for the novel approach in \cite{GBSO}. Stationarity of $\{X_t\}$ is achieved by making $\{\pi_k\}$ the initial of the process at $t=0$. This will be assumed in the remainder of the paper without further reference. 

The marginal pgf's  $\Psi_X(s)$ and $\Psi_\vareps(s)$ of the stationary \inar1 process (\ref{eq_inar1}) and its innovation sequence must satisfy the functional equation
\begin{equation}
\Psi_X(s)=\Psi_X[\Psi_G(s)]\Psi_\vareps (s), \label{fcn_eq_1}
\end{equation}
where $\Psi_G(s)$ is the common pgf of the $G$ variables in (\ref{thin_star}). In the standard approach, $\Psi_G(s)$ is always known. If, in addtion, either $\Psi_X(s)$ or $\Psi_\vareps(s)$, but not both, is specified, the marginal pgf of the innovation process (forward approach), respectively, the marginal pgf of the \inar1 process (backward approach), arises as the solution of the functional equation (\ref{fcn_eq_1}). In \cite{GBSO} both $\Psi_X(s)$ and $\Psi_\vareps(s)$ are specified, and $\Psi_G(s)$ is identified, again via (\ref{fcn_eq_1}), as
\begin{equation}
\Psi_G(s)=\Psi_X^{-1}\Bigl(\frac{\Psi_X(s)}{\Psi_\vareps(s)}\Bigr). \label{fcn_eq_2}
\end{equation}
 
In the remainder of this article, we will designate by $\star$ the thinning operator characterized by (\ref{fcn_eq_2}). We will use a different symbol for other thinning operators.
 
As shown in \cite{GBSO}, the binomial thinning operator with $\Psi_G(s)=1-\a+\a s$, $\a\in (0,1)$ is the solution to (\ref{fcn_eq_2}) for the stationary \inar1 model with a prescribed Poisson($\la$) marginal and a Poisson($\la\ov{\a}$) innovation.
We provide another important example. Let $\ga\in (0,1]$, $\a,\th\in (0,1)$, and $\la>0$. A stationary \inar1 process $\{X_t\}$ whose marginal and innovation distributions have pgf 
\begin{equation}
\Psi_X(s)=\exp\Bigl\{ -\la\Bigl(\frac{1-s}{1-\th s}\Bigr)^\ga\Bigr\} \quad \text{and} \quad \Psi_\vareps(s)=
\exp\Bigl\{ -\la(1-\a^\ga)\Bigl(\frac{1-s}{1-\th s}\Bigr)^\ga\Bigr\}\label{comp_stable_1}
\end{equation}
can be constructed (via (\ref{fcn_eq_2})) with the thinning operator $(\a,\th)\star (\cdot)$ characterized by the linear fractional pgf
\begin{equation}
\Psi_G(s) =\frac{\ov{\a}-(\th-\a)s}{1-\a\th-\ov{\a}\th s}.\label{comp_stable_2}
\end{equation}
The proof is an easy exercise and is therefore omitted. Note that if $0<\ga<1$, then $\{X_t\}$ is heavy-tailed (with an infinite mean), whereas the special case $\ga=1$ gives rise to the stationary \inar1 process with a Polya-Aeppli marginal distribution as well as a Polya-Aeppli innovation. This model was proposed and studied in \cite{AB1} (Sections 2 and 3) by way of the forward approach with the expectation thinning operator $(\b,\th)\odot (\cdot)$ that acts on a \bzp-valued random variable $X$ as follows:
\begin{equation}
(\b,\th)\odot X =\sum_{i=1} Y_i,\label{AB1_operator}
\end{equation}
where $\b\in (0,1)$, $\th\in [0,1)$, and $\{Y_i\}$ is an \iid sequence of \bzp-valued random variables, independent of $X$,  with marginal pgf
\begin{equation}
\Psi_{\b,\th}(s)=1-\b\frac{1-s}{1-\ovb\th s}.\label{AB1_operator2}
\end{equation}
Note the case $\th=0$ corresponds to the binomial thinning operator.

It is easily seen the two thinning operators $(\a,\th)\star (\cdot)$ and $(\b,\th)\odot (\cdot)$, characterized respectively by (\ref{comp_stable_2}) and (\ref{AB1_operator2}), are equivalent through the reparameterization $\b=\frac{\a\ovth}{1-\a\th}$ from $\star$ to $\odot$ and, conversely, $\a=\frac{\b}{1-\ovb\th}$ from $\odot$ to $\star$. Consequently, the stochastic properties of the Polya-Aeppli \inar1 process obtained in \cite{AB1} (Section 3) can be rewritten in terms of the $\star$ thinning operator and the parameters $\a$ and $\th$.   

\section{On the novel geometric \inar1 process}

In \cite{GBSO} (Sections 2.1 and 2.3) the authors introduce a stationary \inar1 process $\{X_t\}$ with prespescified negative binomial marginal and innovation distributions with the following pgf's:
\begin{equation}
\Psi_X(s)=(1+\mu r^{-1}(1-s))^{-r}\quad \text{and}\quad \Psi_\vareps(s)=(1+\ova\mu r^{-1}(1-s))^{-r},\label{GBSO_NB}
\end{equation}
where $\a\in (0,1)$, $r>0$, $\mu>0$, $E(X_t)=\mu$ and $E(\vareps_t)=\ova\mu$. 

Following the authors notation we will write $X_t\sim NB(r,\mu)$ and $\vareps_t\sim NB(r,\ova\mu)$

Using (\ref{fcn_eq_2}), the authors identify the expectation thinning operator $\bth\star(\cdot)$ of (\ref{thin_star}) with $\bth=(\a,\mu,r)$ and pgf.
\begin{equation}
\Psi_G(s)=\frac{1+(r^{-1}\ova\mu-\a)(1-s)}{1+r^{-1}\ova\mu(1-s)}=1-\a\frac{1-s}{1+r^{-1}\ova\mu(1-s)}.\label{NB_thin_star}
\end{equation}

The special case $r=1$ gives rise to the authors' novel geometric \inar1 process with $X_t\sim NB(1,\mu)$, $\vareps_t\sim NB(1, \ova\mu)$ and
\begin{equation}
\Psi_G(s)=\frac{1+(\ova\mu-\a)(1-s)}{1+\ova\mu(1-s) }=1-\a\frac{1-s}{1+\ova\mu(1-s)}.\label{Geo_thin_star}
\end{equation}
In the remainder of the paper, we will refer to the above thinning operators as follows
\begin{equation}\label{thin_Geo_NB}
\bth\star(\cdot)=
\begin{cases}
(\a,\mu)\star(\cdot) & \mbox{if }r=1\\
(\a,\mu,r)\star(\cdot) &\mbox{if }r>0 \mbox{ and }r\ne 1.
\end{cases}
\end{equation}

The novel geometric \inar1 model is fully studied in \cite{GBSO}. Its main properties are given, including  some estimation methods. The results of a Monte Carlo simulation are discussed. The authors conclude with a real-world application. 

We note at this point that the authors introduce an unnecessarily complicated parameterization to describe the common pmf of the \iid sequence $\{G_i\}$.  We propose a simpler and more useful representation that follows from this version of $\Psi_G(s)$,
\begin{equation}
\Psi_G(s)=1-\a q\frac{1-s}{1-\ov{q}s}, \qquad q=(1+\ova\mu)^{-1}, \label{Geo_thin_star2}
\end{equation}
which leads to the $G$-pmf
\begin{equation}
P(G=k)=
\begin{cases}
1-\a q & \mbox{if }k=0\\
(\a q)q\ov{q}^{k-1} &
\mbox{if }k\ge 1.
\end{cases}\label{G_pmf}
\end{equation}

In the remainder of this section we revisit some of the authors' results and give modified versions of a few distributional properties of the novel geometric \inar1 process by dropping some unnecessary assumptions.

First, we show that the constraint $0< \a < \frac{\mu}{1+\mu}$ in Lemma 2.2 in \cite{GBSO} is not needed to obtain the conditional probability of $(\a,\mu) \star X$ given $X=x$.  

\begin{lemma}
 Assume that $X$ is a \bzp-valued random variables. Then, the conditional probability of $(\a,\mu) \star X$ given $X=x$ is 
\begin{equation}
p_k^{(x)}=P((\a,\mu) \star X=k|X=x)=
\begin{cases}
(1-\a q)^{x} & \mbox{if }k=0\\
\displaystyle\sum_{i=1}^{\min(k,x)}A_i^{(x)} (\a q) B_i^{(k)}(q) &
\mbox{if }k\ge 1,
\end{cases}\label{conv_G}
\end{equation}
where, for $n\ge 1$, $0\le i\le n$, $1\le l\le n$, and $0<y<1$,
\begin{equation}
A_i^{(n)}(y)=\binom{n}{i} y^i(1-y)^{n-i} \quad \text{and} \quad B_l^{(n)}(y)=\binom{n-1}{l-1} y^l(1-y)^{n-l}.
\label{A_and_B}
\end{equation} 
\end{lemma}

\begin{proof} Note that the conditional pgf of $(\a,\mu) \star X$ given $X=x$ is $\Psi_G(s)^x$. It is easily seen that the modified parameterization of $\Psi_G(s)$ (cf. (\ref{Geo_thin_star2}) and (\ref{G_pmf})) can be rewritten as 
\begin{equation}
\Psi_G(s) =1-\a q+\a qG_q(s), \label{Psi_G_Geo}
\end{equation}
where $G_q(s)$ is the pgf of the shifted geometric($q$) distribution of (\ref{shifted_Geo}). We have for every $i\ge 1$,
\begin{equation}
G_q(s)^i=\sum_{k=i}^\infty \binom{k-1}{i-1}\,q^i\ov{q}^{k-i}s^i.\label{Truncated_NB}
\end{equation}
An application of the binomial theorem and a standard power series argument lead to (\ref{conv_G}). 
\end{proof}

Next, we modify the formula for the first-order transition probability $p_{ij}=P(X_{t+1}=j|X_t=i)$ given in Lemma 3.1 in \cite{GBSO} and show again that the assumption $0<\a<\frac{\mu}{\mu+1}$ is not needed. 

\begin{proposition}
Let $\{X_t\}$ be a stationary \inar1 process with prespecified $NB(1,\mu)$ marginal and $NB(1,\ova\mu)$ innovation distributions. Then,   
\begin{equation}
p_{ij}=
\begin{cases}
q\ov{q}^j \qquad \mbox{if }i=0\\
\displaystyle A_0^{(i)}(\a q)B_1^{(j+1)}(q)+\sum_{k=1}^j B_1^{(j-k+1)}(q)\sum_{l=1}^{\min(i,k)} A_l^{(i)}(\a q)B_l^{k}(q) & \mbox{if }i\ge 1,
\end{cases}
\label{trans_order_1}
\end{equation}
\end{proposition}
\begin{proof}
First, we note  $p_{ij}$ results from the convolution of the pmf $(p_k^{(i)}, k\ge 0)$ of (\ref{conv_G}) and the pmf $(q\ov{q}^k, k\ge 0)$ of $\vareps_t$. It is clear that $p_{0j}=P(\vareps_t=j)$. For $i\ge 1$, we have $p_{ij}=p_0^{(i)}(q\ov{q}^j)+\sum_{k=1}^j p_k^{(i)}(q\ov{q}^{j-k})$, which yields (\ref{trans_order_1}), by way of (\ref{conv_G}) and (\ref{A_and_B}).
\end{proof}

We proceed to extend (\ref{trans_order_1}) to the $h$-th order transition probability from $X_t$ to $X_{t+h}$. 

For $h\ge 0$, we define the sequence $\Psi_G^{(h+1)}(s)=\Psi_G\bigl[\Psi_G^{(h)}(s)\bigr]$ with $\Psi_G^{(0)}(s)=s$ for $\Psi_G(s)$ of (\ref{Geo_thin_star}). 

We gather a few useful results in the following lemma. The proof of each assertion relies on a straightforward induction argument. We omit the details.

\begin{lemma} The following assertions hold.
\begin{enumerate}
\item[(i)] For every $h\ge 1$,
\begin{equation}
\Psi_G^{(h)}(s)=1-\a^hq_h\frac{1-s}{1-(1-q_h)s}, \quad q_h=(1+(1-\a^h)\mu)^{-1}. \label{Psi_G_h}
\end{equation}
\item[(ii)] The $h$-fold composition of the operator $(\a,\mu)\star(\cdot)$ ($h\ge 1$) satisfies 
\begin{equation}
(\a,\mu)\star(\a,\mu)\star\cdots\star(\a,\mu)*Y\dover(\a^h,\mu)\star Y \quad (h\ge 1) \label{h_fold_thin_star}
\end{equation}
for any \bzp-valued random variable $Y$. Moreover, $\Psi_{(\a^h,\mu)\star Y}(s)=\Psi_Y\bigl(\Psi_G^{(h)}(s)\bigr)$.
\item[(iii)] The \inar1 process $\{X_t\}$ of (\ref{eq_inar1}), with  $\bth\star (\cdot)=(\a,\mu)\star(\cdot)$, satisfies the property
\begin{equation}
X_{t+h}\dover (\a^h,\mu)\star X_t +\sum_{j=1}^h(\a^{h-j},\mu)\star\vareps_{t+j}. \label{h_forward_eq}
\end{equation}
for every $h\ge 1$ (note, by convention, $(1,\mu)\star Y=Y$).
\end{enumerate}
\end{lemma}

We now derive the $h$-th order transition probability $p_{ij}^{(h)}=P(X_{t+h}=j|X_t=i)$, $h\ge 1$, of the novel geometric \inar1 process. Note the  more complicated formula obtained in \cite{GBSO}, Proposition 3.3, requires the unnecessary assumption $\a^h < \frac{\mu}{\mu+1}$.

\begin{proposition}
Let $\{X_t\}$ be a stationary \inar1 process with prespecified $NB(1,\mu)$ marginal and $NB(1,\ova\mu)$ innovation distributions. Then for every $h\ge1$, 
\begin{equation}
p_{ij}^{(h)}=
\begin{cases}
 q_h(1-q_h)^j \qquad \mbox{if }i=0\\
\displaystyle A_0^{(i)}(\a^h q_h)B_1^{(j+1)}(q_h)+\sum_{k=1}^j B_1^{(j-k+1)}(q_h)\sum_{l=1}^{\min(i,k)} A_l^{(i)}(\a^h q_h)B_l^{k}(q_h)
 &\mbox{if }i\ge 1.
\end{cases}\label{trans_order_h}
\end{equation}
Note $p_{ij}^{(1)}=p_{ij}$, with $p_{ij}$ of (\ref{trans_order_1}) and $q_1=q$.
\end{proposition}

\begin{proof} First, we note that for every $h\ge 1$,
\begin{equation}
\Psi_X\bigl(\Psi_G^{(h)}(s)\bigr)={1+(1-\a^h)\mu(1-s)\over 1+\mu(1-s)} \text{ and }  \Psi_\vareps\bigl(\Psi_G^{(h)}(s)\bigr)={1+(1-\a^h)\mu(1-s)\over 1+(1-\a^{h+1})\mu(1-s)}, \label{Psi_X_vareps_h}
\end{equation}
where $\Psi_X(s)$ and $\Psi_\vareps(s)$ are the marginal pgf's of the geometric \inar1 process and its innovation (see (\ref{GBSO_NB}) with $r=1$). Letting $\vareps_t^{(h)}=\sum_{j=0}^h(\a^{h-j},\mu)\star\vareps_{t+j}$, it follows by (\ref{h_fold_thin_star}),  (\ref{Psi_X_vareps_h}), and independence, that the pgf of $\vareps_t^{(h)}$ is
\begin{equation}
\Psi_{\vareps^{(h)}}(s)=\prod_{j=1}^h \Psi_\vareps\bigl(\Psi_G^{(h-j)}(s)\bigr)=\frac{q_h}{1-(1-q_h)s}, \label{pgf_vareps_h}
\end{equation}
with $q_h$ of (\ref{Psi_G_h}).That is, $\vareps_t^{(h)}\sim NB(1,(1-\a^h)\mu)$ with pmf $\bigl(q_h(1-q_h)^k, k\ge 0\bigr)$. Noting by Lemma 3.2-(i) that $\Psi_G^{(h)}(s)=1-\a^hq_h+\a^hq_hG_{q_h}(s)$, where $G_{q_h}(s)$ is the pgf of a shifted geometric($q_h$) distribution. The exact same argument used to derive (\ref{conv_G}) in  the case $h=1$ leads to
\begin{equation}
p_k^{(x)}(h)=P((\a^h,\mu) \star X=k|X=x)=
\begin{cases}
(1-\a^h q_h)^{x} & \mbox{if }k=0\\
\displaystyle\sum_{i=1}^{\min(k,x)}A_i^{(x)} (\a^h q_h) B_i^{(k)}(q_h) &
\mbox{if }k\ge 1,
\end{cases}\label{conv_G_h}
\end{equation}
where the functions $A_i^{(x)}$ and $B_i^{(k)}$ are given in (\ref{A_and_B}). Therefore, by (\ref{h_forward_eq})  $p_{ij}^{(h)}$ results from the convolution of the pmf's  $(p_k^{(i)}(h), k\ge 0)$ and $\bigl(q_h(1-q_h)^k, k\ge 0\bigr)$. We obtain (\ref{trans_order_h}) by using the same argument that led to (\ref{trans_order_1}) in the case $h=1$.
\end{proof}

In the next section we give two simple formulas for the conditional pgf and the conditional variance of $X_{t+h}$ given $X_t$ (by setting $r=1$ in Proposition 4.2, (i) and (iii)).

\section{On the novel negative binomial \inar1 process} 

The aim of this section is to extend the results obtained in \cite{GBSO} for the novel geometric \inar1 process to the novel negative binomial \inar1 process $\{X_t\}$ (see Section 2.3 in \cite{GBSO}), where $X_t\sim NB(r,\mu)$, $\vareps_t\sim NB(r,\ova\mu)$, $\mu=E(X_t)$, and where the novel thinning operator of (\ref{thin_star}) is $(\a,\mu,r)\star(\cdot)$ is characterized by the pgf $\Psi_G(s)$ of (\ref{NB_thin_star}).

We recall that in \cite{AB1} the authors used a different parameterization for the negative binomial distribution. For a \bzp-valued random variable $X$ we write $X\sim NB1(r,\th)$ with parameters $0<\th<1$ and $r>0$ (as opposed to $X\sim NB(r,\mu)$) if its pgf $\Psi_X(s)$ and its pmf $\{P(X=k\}$ are as follows:
\begin{equation}
\Psi_X(s)=\Bigl(\frac{\ovth}{1-\th s}\Bigr)^r \quad \text{and} \quad P(X=k)={\Ga(k+r)\over k!\Ga(r)}\ov\th^r \th^k \quad(k\ge 0).\label{NB1_dist}
\end{equation}
The mean and variance of $X$ are $E(X)=\frac{r\th}{\ov\th}$ and $Var(X)=\frac{r\th}{\ov\th^2}$.

One can show that all the distributional properties of the novel geometric \inar1 process ($r=1$) obtained in Section 3 in \cite{GBSO} extend to the novel negative binomial \inar1 process by a modification of the proofs therein and in section 3 above. Instead, we show via a suitable reparameterization that the $(\a,\mu,r)\star(\cdot)$ operator acts as the $(\b,\th)\odot(\cdot)$ operator introduced in \cite{AB1} and \cite{AB2} and characterized by equations (\ref{AB1_operator}) and (\ref{AB1_operator2}) above, thus leading to the conclusion that the novel $NB(\a,\mu)$ \inar1 process in \cite{GBSO} is equivalent to the $NB1(r,\th)$ \inar1 process in \cite{AB1}. As a consequence,  the distributional properties of the novel model will be derived from those established for its counterpart in \cite{AB1}.

We start out with a useful lemma. Its proof is an easy exercise and is therefore omitted. 

\begin{lemma}
Let $r>0$ and $\b,\th\in (0,1)$.
\begin{enumerate}
\item[(i)] The thinning operator  $(\b,\th)\star(\cdot)$ (cf.  (\ref{thin_star}) and and (\ref{fcn_eq_2}) ) obtained by prespecifying $X_t\sim NB1(r,\th)$ and $\vareps_t\sim NB1(r,\ovb\th)$ is characterized by the pgf $\Psi_G(s)=\Psi_{\b,\th}(s)$ of (\ref{AB1_operator2}). Therefore, the thinning operators $(\b,\th)\odot(\cdot)$ and $(\b,\th)\star(\cdot)$ are identical and can be used interchangeably.
\item[(ii)] The thinning operator $(\a,\mu,r)\star(\cdot)$ (cf. (\ref{NB_thin_star}) and (\ref{thin_Geo_NB})) can take the form of the $(\b,\th)\star(\cdot)$ (or $(\b,\th)\odot(\cdot)$) thinning operator, with
\begin{equation}\label{from_star_to_odot}
\b={\a r\over r+\ova\mu} \quad \text{and} \quad \th={\mu\over \mu+r}'
\end{equation}
The converse statements holds with
\begin{equation}\label{from_odot_to_star}
\a={\b \over 1-\ovb\th} \quad \text{and} \quad \mu={\th r\over \ov{\th}}
\end{equation}
\end{enumerate}
\end{lemma}

Next, we recall a few basic results from \cite{AB1}. 

The \inar1 process $\{X_t\}$ in \cite{AB1} is defined by the equation
\begin{equation}
X_{t+1} =(\b,\th) \odot X_t + \vareps_{t+1},\label{eq_inar_AB1}
\end{equation}
under the same assumptions as the ones set for the \inar1 process defined with the thinning operator $\bth\star(\cdot)$ (see (\ref{eq_inar1})).

The \inar1 model (\ref{eq_inar_AB1}) with an $NB1(r,\th)$ marginal, $r>0$ and $0<\th <1$, was constructed using the forward approach. The stationarity of the process follows from Proposition 2.5 and Lemma 4.1 in \cite{AB1} (see discussion in Section 2). Its innovation sequence is necessarily $NB1(r,\ovb\th)$ and it is time reversible (Corollary 4.3 in \cite{AB1}) . The latter property is in fact a characterization of the process among all the \inar1 processes governed by (\ref{eq_inar_AB1}) (see Proposition  5.1 in \cite{AB1}).

We extend the function $B_l^{(n)}(y)$ in (\ref{A_and_B}) to positive numbers $n$ and $l$ such that $n\ge l$ as follows:
\begin{equation}
B_l^{(n)}(y)={\Ga(n)\over \Ga(l)\Ga(n-l+1)}y^l(1-y)^{n-l}.\label{B_extended}
\end{equation}

We will also use the following formula for the $NB1(r,\th)$ pmf (\ref{NB1_dist}):
\begin{equation}
P(X=k)=B_r^{(k+r)}(\ovth). \label{NB1_B_function}
\end{equation}

In the next proposition we list a few basic properties of the novel negative binomial \inar1 process. The assertions are a mere translation in terms of $\a$ and $\mu$ (aided by Lemma 4.1 and (\ref{NB1_B_function})) of the distributional and correlation properties of the ($NB1(r,\th)$ \inar1 process of (\ref{eq_inar_AB1}) that are summarized in Proposition 4.2 and Corollary 4.3 in \cite{AB1}. The details are omitted.

\begin{proposition}
Let $\a\in(0,1)$ and $\mu, r>0$. Let $\{X_t\}$ be a stationary \inar1 process driven by the thinning operator $(\a,\mu,r)\star(\cdot)$ with a $NB(r,\mu)$ marginal distribution and a $NB(r,\ova\mu)$ innovation distribution. Then, 
\begin{enumerate}
\item[(i)] $E(X_{t+1}|X_t)=\a X_t+\ova\mu$;
\item[(ii)] $Var(X_{t+1}|X_t)=\a\ova (2\mu r^{-1}+1) X_t+\ova\mu(1+r^{-1}\ova\mu)$;
\item[(iii)] The auto-correlation function $\{\rho_k\}$ of $\{X_t\}$ is $\rho_k=\a^k$, $k\ge1$. 
\item[(iv)] The transition probability $\wtilde{p}_{ij}=P(X_{t+1}=j|X_t=i)$ is
\begin{equation}\label{trans_order_1_NB}
\wtilde{p}_{ij}=
\begin{cases}
B_r^{j+r}(\wtilde{q}) \qquad \mbox{if }i=0\\
\displaystyle A_0^{(i)}(\a \wtilde{q}) B_r^{(j+r)}(\wtilde{q})
+\sum_{k=1}^j B_r^{(j-k+r)}(\wtilde{q})\sum_{l=1}^{\min(i,k)} A_l^{(i)}(\a \wtilde{q})B_l^{(k)}(\wtilde{q}),
&\mbox{if }i\ge 1.
\end{cases}
\end{equation}
where $\wtilde{q}={r\over r+\ova\mu}$ and the functions $A_i^{(x)}$ and $B_i^{(k)}$ are given in (\ref{A_and_B}) and (\ref{B_extended}).
\item[(v)] The joint pgf $\varphi(s_1,s_2)$ of $(X_t,X_{t+1})$ is
\begin{equation}\label{bv_pgf}
\varphi(s_1,s_2)=\Biggl[r^{-2}\Bigl( (r+\mu)(r+\ova\mu)-\ova\mu(r+\mu)(s_1+s_2)
+\mu(\ova\mu-r\a) s_1s_2\Bigr)\Biggr]^{-r}, 
\end{equation}
leading to the conclusion that $\{X_t\}$ is time reversible.
\end{enumerate}
\end{proposition}

Irreducibility and stationarity of the novel $NB(r,\mu$ \inar1 process lead to its ergodicity, as seen in the following proposition.
 
\begin{proposition}
Let $\a\in(0,1)$ and $\mu, r>0$. Let $\{X_t\}$ be a stationary \inar1 process driven by the thinning operator $(\a,\mu,r)\star(\cdot)$ with a $NB(r,\mu)$ marginal distribution and a $NB(r,\ova\mu)$ innovation distribution. Then $\{X_t\}$ is a homogenous Markov chain that is irreducible, positive recurrent and aperiodic, and therefore ergodic.
\begin{proof}
We need to show that $\{X_t\}$ is irreducible, positive recurrent and aperiodic. By Proposition 4.1 and (\ref{trans_order_1_NB}), $\wtilde{p}_{ij}>0$ for every $i,j\ge 0$. Therefore, $\{X_t\}$ is irreducible and aperiodic. Stationarity implies the homogeneity property. The stationary distribution of the process is $NB(r,\mu)$. Therefore, every state of $\{X_t\}$ is positive recurrent  by Theorem 3.2.6, p. 121 in \cite{BREM}.
\end{proof} 
\end{proposition}

Next, we derive the $h$-step ahead distributional properties of $\{X_t\}$ for $h\ge1$. This requires deriving a few additional formulas for the $NB1(r,\th)$ \inar1 process.

We will need the following properties of the operator $(\b,\th)\odot(\cdot)$ (see equations (2.9) and (2.10) in \cite{AB1}).

Let $\b\in (0,1)$ and $\th\in [0,1)$. The $h$-fold composition $h\ge 1$ of the operator $(\b,\th)\odot(\cdot)$ satisfies 
\begin{equation}
(\b,\th)\odot(\b,\th)\odot\cdots\odot(\b,\th)*Y\dover(\b_h,\th)\odot Y \text{ and } \Psi_{(\b_h,\th)\odot Y}(s)=\Psi_Y\bigl(\Psi_{\b_h,\th}(s)\bigr), \label{h_fold_thin_odot}
\end{equation}
where $\Psi_{\b_h,\th}(s)$ is as in (\ref{AB1_operator2}) and the sequence $\{\b_h\}$ is defined recursively by
\begin{equation}
\b_h=\frac{\b^h\ovth}{(1-\ovb\th)^h-\b^h\th}; \qquad \ovb_h=1-\b_h=\frac{(1-\ovb\th)^h-\b^h}{(1-\ovb\th)^h-\b^h\th}. \label{h_fold_thin_odot_2}
\end{equation}

By Proposition 2.2 in \cite{AB1}, the conditional distribution of $(\b_h,\th) \odot X_t$ given $X_t=x$ has pgf $\Psi_X\bigl(\Psi_{\b_h,\th}(s)\bigr)^x$, where
\begin{equation}\label{Psibh}
\Psi_{\b_h,\th}(s)=\Bigl(1-\b_h{1-s\over 1-\ovb_h\th s}\Bigr)^x,
\end{equation}  
and pmf
\begin{equation}\label{pkxh}
\wtilde{p}_k^{(x)}(h)=P((\b_h,\th) \odot X_t=k|X=x)=
\begin{cases}
\ovb_h^{\ x} & \mbox{if }k=0\\
\displaystyle\sum_{i=1}^{\min(k,x)}A_i^{(x)} (\b_h) B_i^{(k)}(1-\ovb_h\th) &
\mbox{if }k\ge 1,
\end{cases}
\end{equation}

where the functions $A_i^{(x)}$ and $B_i^{(k)}$ are given in (\ref{A_and_B}).

\begin{lemma}
Let $\{X_t\}$ be the $NB1(r,\th)$ stationary \inar1 process with an $NB1(r,\ovb\th)$ innovation distribution. The following assertions hold.
\begin{enumerate}
\item[(i)] For every $h\ge 1$, the $h$-step ahead equation of $\{X_t\}$ is
\begin{equation}
X_{t+h}\dover (\b_h,\th)\odot X_t+\wtilde{\vareps}_t^{(h)}; \qquad \wtilde{\vareps}_t^{(h)}=\sum_{j=1}^h (\b_{h-j},\th)\odot \vareps_{t+j}. \qquad \label{h_step_inar_AB1_1}
\end{equation}

\item[(ii)] For every $t\ge 1$, $\wtilde{\vareps}_t \sim NB1(r,\ovb_h\th)$.
\item[(iii)]  The conditional pgf $\Psi_{X_{t+h}|X_t}(s)$ of $X_{t+h}$ given $X_t$ is 
\[
\Psi_{X_{t+h}|X_t}(s)=\Biggl(1-\b_h{1-s\over 1-\ovb_h\th s}\Biggr)^{X_t}\Biggl({1-\ovb_h\th\over 1-\ovb_h\th s}\Biggr)^r,
\] 
\item[(iv)] The conditional mean and variance of $X_{t+}h$ given $X_t$
\[
E(X_{t+h}|X_t)={\b_h\over 1-\ovb_h\th}X_t+{r\ovb_h\th\over 1-\ovb_h\th}
\]
and
\[
Var(X_{t+h}|X_t)={\b_h\ovb_h(1+\th)\over (1-\ovb_h\th)^2}X_t+{r\ovb_h\th\over (1-\ovb_h\th)^2}.
\]
\item[(v)] The transition probability $\wtilde{p}_{ij}^{(h)}=P(X_{t+h}=j|X_t=i)$ is $\wtilde{p}_{0j}^{(h)}=B_r^{(j+r)}(1-\ovb_h\th)$ for $i=0$, and for $i\ge1$,
\begin{plain}
\begin{equation}\label{trans_order_h_NB}
\eqalign{
\widetilde{p}_{ij}^{(h)}=A_0^{(i)}(\b_h)B_r^{(j+r)}(1-\ovb_h\th)&+\cr
\sum_{k=1}^j B_r^{(j-k+r)}(1-\ovb_h\th)&\sum_{l=1}^{\min(i,k)} A_l^{(i)}(\b_h)B_l^{(k)}(1-\ovb_h\th).\cr}
\end{equation}
\end{plain}
\end{enumerate}
\end{lemma}

\begin{proof} For (i), the pgf version of (\ref{eq_inar_AB1}) is $\Psi_X(s)=\Psi_X\bigl(\Psi_{\b,\th}(s)\bigr)\Psi_\vareps(s)$,  with $\Psi_{\b,\th}(s)$ of (\ref{Psibh}), which implies by stationarity that
\begin{equation}
\Psi_X(s)=\Psi_X\bigl(\Psi_{\b_h,\th}(s)\bigr)\prod_{j=1}^h \Psi_\vareps\bigl(\Psi_{\b_{h-j},\th}(s)\bigr)\qquad (h\ge 1).\label{h_step_inar_pgf}
\end{equation}
which in turn implies the $h$-step ahead equation 
\begin{equation}\label{h_step_inar_AB1_2}
X_{t+h}\dover (\b_h,\th)\odot X_t+\sum_{j=1}^h (\b_{h-j},\th)\odot \vareps_{t+j}.
\end{equation}
This proves the representation (\ref{h_step_inar_AB1_1}). We note that for every $l\ge1$, the random variable $(\b_l,\th)\odot \vareps_t$ (recall $\vareps_t \sim NB1(r,\ovb\th)$) has pgf 
\[
\Psi_\vareps(\Psi_{\b_l,\th}(s))=\Biggl({1-\ovb_{l+1}\th\over 1-\ovb_l\th}\cdot {1-\ovb_l\th s\over 1-\ovb_{l+1}\th s}\Biggr)^r.
\]
It follows that the pgf of $\wtilde{\vareps}_t$, $\Psi_{\wtilde{\vareps}}(s)=\prod_{j=1}^h \Psi_\vareps\bigl(\Psi_{\b_{h-j},\th}(s)\bigr)$ (with $\b_0=1$) simplifies to $\Psi_{\wtilde{\vareps}}(s)=\Bigl({1-\ovb_h\th\over 1-\ovb_h\th s}\Bigr)^r$, which proves (ii). Part (iii) is a straightforward consequence of equation (\ref{h_step_inar_AB1_1}) and the independence assumption. The formulas in (iv) are derived by computing the first and second derivatives of $\Psi_{X_{t+h}|X_t}(s)$ at $s=1$. For (v) we note the transition probability $\wtilde{p}_{ij}^{(h)}=P(X_{t+h}=j|X_t=i)$ results from the convolution of $\{\wtilde{p}_k^{(i)}(h)\}$ of (\ref{pkxh}) and the pmf of $\wtilde{\vareps}_t$. The formula for the latter is (see (\ref{B_extended}) and (\ref{NB1_B_function})) 
\[
P(\wtilde{\vareps}_t=k)= \frac{\Ga(k+r)}{\Ga(k+1)\Ga(r)}(1-\ovb_h\th)^r (\ovb_h\th)^k=B_r^{(k+r)}(1-\ovb_h\th).
\]
An application of the formula for the convolution of two pmf's yields (\ref{trans_order_h_NB}).
\end{proof}

We now present the results above in terms of the parameterization used in \cite{GBSO}.  The basic formulas used in the translation are (\ref{from_star_to_odot}) and (\ref{from_odot_to_star}). In addition,  we define 
$\wtilde{q}_h={r\over r+(1-\a^h)\mu}$. The following equations are easily shown to hold:
\[
1-\ovb_h\th=\wtilde{q}_h; \quad \b_h=\a^h\wtilde{q}_h; \quad; \ovb_h\th=1-\wtilde{q}_h; \quad \ovb_h=1-\a^h\wtilde{q}_h.
\]

\begin{proposition}
Let $\a\in(0,1)$, $\mu, r>0$, and $h\ge 1$. Let $\{X_t\}$ be a stationary \inar1 process driven by the thinning operator $(\a,\mu,r)\star(\cdot)$ with a $NB(r,\mu)$ marginal distribution and a $NB(r,\ova\mu)$ innovation distribution. Then, 
\begin{enumerate}
\item[(i)] The conditional pgf $\Psi_{X_{t+h}|X_t}(s)$ of $X_{t+}h$ given $X_t$ is
\[
\Psi_{X_{t+h}|X_t}(s)=\Biggl(1-\a^h\wtilde{q}_h{1-s\over 1-(1-\wtilde{q}_h) s}\Biggr)^{X_t}\Biggl({\wtilde{q}_h\over 1-(1-\wtilde{q}_h) s}\Biggr)^r;
\]
\item[(ii)] $E(X_{t+h}|X_t)=\a^h X_t+\mu(1-\a^h)$;
\item[(iii)] $Var(X_{t+h}|X_t)=(2\mu r^{-1}+1)\a^h(1-\a^h) X_t+\mu(1-\a^h)(1+(1-\a^h)\mu r^{-1})$;
\item[(iv)] The transition probability $\wtilde{p}_{ij}^{(h)}=P(X_{t+h}=j|X_t=i)$ is
\begin{equation}\label{trans_order_1_NB_2}
\wtilde{p}_{ij}^{(h)}=
\begin{cases}
B_r^{j+r}(\wtilde{q}_h)\quad \mbox{if }i=0\\
\displaystyle A_0^{(i)}(\a^h\wtilde{q}_h)B_r^{(j+r)}(\wtilde{q}_h)+
\sum_{k=1}^j B_r^{(j-k+r)}(\wtilde{q}_h)\sum_{l=1}^{\min(i,k)} A_l^{(i)}(\a^h\wtilde{q}_h)B_l^{(k)}(\wtilde{q}_h)
&\mbox{if }i\ge 1.
\end{cases}
\end{equation}
\end{enumerate}
\end{proposition}
As noted earlier, setting $r=1$ in (i) and (iii) yields simpler formulas than the ones given in \cite{GBSO} for the conditional pgf and conditional variance of $X_{t+h}$ given $X_t$  of the novel geometric \inar1 process.

We conclude with a weak convergence result that leads to a representation of the novel negative binomial \inar1 process by moving-average process of infinite order denoted MA($\infty$).

\begin{proposition}
Let $\a\in(0,1)$, $\mu, r>0$ and $(\b,\th)$ as in (\ref{from_star_to_odot}). Let $\{X_t\}$ be a stationary \inar1 process driven by the thinning operator $(\a,\mu,r)\star(\cdot)$ with a $NB(r,\mu)$ marginal distribution and a $NB(r,\ova\mu)$ innovation distribution. Then,
\begin{equation}
X_t\dover \sum_{j=0}^\infty (\b_j,\th)\odot\vareps_{t-j}\dover\sum_{j=0}^\infty (\a^j,\mu,r)\star\vareps_{t-j} \qquad (t\in \mathbb{Z}). \label{MA_infty}
\end{equation}
\end{proposition}

\begin{proof} Since $\{X_t\}$ is a stationary Markov chain, it can be extended to a doubly-infinite stationary sequence $(X_t, n \in \mathbb{Z}$), as seen in \cite{BREI}, Proposition 6.5, p. 105. Therefore, going backwards $h$ steps (contrast with (\ref{h_step_inar_AB1_2})), we have $X_t\dover (\b_h,\th)\odot X_{t-h}+\sum_{j=1}^h (\b_j,\th)\odot \vareps_{t-j}$. In terms 
of pgf's, we have 
\[
\Psi_X(s)=\Psi_X\bigl(\Psi_{\b_h,\th}(s)\bigr)\prod_{j=0}^h\Psi_\vareps\bigl(\Psi_{\b_j,\th}(s)\bigr).
\]
Since $\lim_{h\to\infty}\b_h=0$, $\lim_{h\to 0}\Psi_{\b_h,\th}(s)=1$, which implies
\[
\Psi_X(s)=\prod_{j=0}^\infty \Psi_\vareps\bigl(\Psi_{\b_j,\th}(s)\bigr),
\]
which in turn implies the first equation in (\ref{MA_infty}). The second equation is just the translation of the first in terms of the operator $(\a,\mu,r)\star (\cdot)$.
\end{proof}

\section{Estimation}

In this section we discuss the standard estimation procedures used for \inar1 processes. We describe in detail the conditional least squares and the Yule-Walker methods for the means of the novel negative binomial process and its innovation sequence.  We give a brief outline of the numerical approach to obtain their conditional maximum likelihood estimators. We also include  the conditional least squares for the variances of the process and its innovation. We rely principally on results obtained by several authors for branching processes with immigration.

\subsection{The conditional least squares method for the means}

We start out with the conditional least squares (CLS) method developed in \cite{KLNE} for stochastic processes and in particular for branching processes with immigration (\cite{QUI}, \cite{WIN} and \cite{HESE}). We rely again on the fact that the \inar1 process described by the equations (\ref{thin_star})--(\ref{eq_inar1}) process is interpretable as a branching process with immigration. In the case of the novel negative binomial \inar1 process $\{X_t\}$, $X_t$ is the population size at time $t$ whose offsrpring distribution has pg.f $\Psi_{\b,\th}(s)$ of (\ref{AB1_operator2}), used with the operator $(\b,\th)\odot (\cdot)$, or, equivalently, $\Psi_G(s)$ of (\ref{NB_thin_star}), used with the operator $(\a,\mu, r)\star (\cdot)$. The immigration process is $\vareps_t$ with marginal distribution $NB1(r, \ovb\th)$, or equivalently, $NB(r,\ova\mu)$. Therefore, the results in Section 5 in \cite{KLNE} and in \cite{WIN} apply directly. 

We need (recall) some notation for the various distributional summaries of $X_t$ and $\vareps_t$. Most of the the formulas will be given in both the parametrization in \cite{AB1} and the one in \cite{GBSO}. We will also use the notation $\mu_k^{(Y)}=E[(Y-\mu_Y)^k]$ to designate the $k$-th central moment of a $\mathbb{Z}_+$-valued random variable $Y$.  That will be needed in this section for $k\in \{2,3,4\}$.

The mean and variance of the $G$-variable in (\ref{thin_star}) are
\[
\mu_G=\frac{\b}{1-\ovb\th}=\a \quad \text{and}\quad \sig_G^2=\frac{\b\ovb(1+\th)}{(1-\ovb\th)^2}=\a\ova(2\mu r^{-1}+1).
\]
The third and fourth central moments of the $G$-variable have lengthy formulas. They can be computed recursively via the formula
\[
\mu_k^{(G)}=\frac{\ovb(-\b)^k}{(1-\ovb\th)^k}+\b\sum_{j=0}^k\binom{k}{j}\mu_j^{(G_1)},
\]
where $G_1\sim NB1(1,\ovb\th)$ and whose central moments are given in \cite{JKK}, p. 208. The details are omitted.

The mean and variance of $\vareps_t$ are
\[
\mu_\vareps=\frac{r\ovb\th}{1-\ovb\th}=\ova\mu \quad \text{and}\quad \sig_\vareps^2=\frac{r\ovb\th}{(1-\ovb\th)^2}=\ova\mu(\ova\mu r^{-1}+1),
\]
and its third and fourth central moments are (see again \cite{JKK}, p. 208)
\[
\mu_3^{(\vareps)}=\frac{r\ovb\th(1+\ovb\th)}{(1-\ovb\th)^3}=\mu_\vareps(\mu_\vareps r^{-1}+1)(2\mu_\vareps r^{-1}+1).
\]
and
\[
\mu_4^{(\vareps)}=\frac{r\ovb\th}{(1-\ovb\th)^2}\Bigl(1+\frac{3\ovb\th(r+2)}{(1-\ovb\th)^2}\Bigr)=
\mu_\vareps(\mu_\vareps r^{-1}+1)\Bigl[1+3\mu_\vareps(\mu_\vareps r^{-1}+1)(1+2r^{-1})\Bigr].
\]
The mean and variance of $X_t$ are
\[
\mu=\frac{r\th}{\ovth}\quad \text{and}\quad \sig^2=\frac{r\th}{\ovth^2}=\mu(\mu r^{-1}+1),
\]
and its third and fourth central moments are
\[
\mu_3^{(X)}=\frac{r\th(1+\th)}{\ovth^3}=\mu(\mu r^{-1}+1)(2\mu r^{-1}+1).
\]
and
\[
\mu_4^{(X)}=\frac{r\th}{\ovth^2}\Bigl(1+\frac{3\th(r+2)}{\ovth^2}\Bigr)=
\mu(\mu r^{-1}+1)\Bigl[1+3\mu(\mu r^{-1}+1)(1+2r^{-1})\Bigr].
\]
We will also need 
\begin{equation}\label{c_squared}
c^2=\mu\sig_G^2+\sig_\vareps^2=\a\ova\sig^2.
\end{equation}

We note that the distributional summaries $(\mu, \sig^2, \mu_3^{(X)}, \mu_4^{(X)})$ of the marginal distribution of $\{X_t\}$ can be written in terms of their counterparts for the $\vareps_t$ and $G$ random variables. This is true in general (see Remark 3.2 in \cite{HESE} and formulas therein).
 
The CLS estimators of the parameters $\a$ and $\mu_\vareps$ of the novel negative \inar1 process are obtained by minimizing the conditional sum of squares 
\begin{equation}\label{Q_CLS}
Q_n(\a,\mu_\vareps)=\sum_{t=1}^n(X_t-\a X_{t-1} -\mu_\vareps)^2.
\end{equation}.
The solutions are
\begin{equation}\label{CLS_Estimators}
\widehat\a_{n,cls}=\frac{n\sum_{t=1}^n X_{t-1}X_t-\sum_{t=1}^n X_{t-1}\sum_{t=1}^n X_t}{n\sum_{t=1}^n X_{t-1}^2-\bigl(\sum_{t=1}^n X_{t-1}\bigr)^2}.
\end{equation}
and 
\begin{equation}\label{CLS_Estimators_2}
\widehat{(\mu_\vareps)}_{n,cls}=\frac{\sum_{t=1}^n X_{t-1}^2\sum_{t=1}^n X_t-\sum_{t=1}^n X_{t-1}\sum_{t=1}^n X_{t-1}X_t}{n\sum_{t=1}^n X_{t-1}^2-\bigl(\sum_{t=1}^n X_{t-1}\bigr)^2}.
\end{equation}
or, more compactly,
\begin{equation}\label{CLS_estimators_2p}
\widehat{(\mu_\vareps)}_{n,cls}=\frac{1}{n}\Biggl[\sum_{t=1}^n X_t-\widehat\a_{cls}\sum_{t=1}^n X_{t-1}\Biggr]
\end{equation}
The following result is a direct consequence of Theorem 3.1 and Theorem 3.2 in \cite{KLNE}, along with their interpretation in the authors' Section 5. The proof is omitted. Note the basic assumption for part (ii) is that $\mu_4^{(G)}$ and $\mu_4^{(\vareps)}$ be finite holds in this case.

\begin{proposition} Let $\widehat\a_{n,cls}$ and $\widehat{(\mu_\vareps)}_{n,cls}$ be the CLS estimators  (\ref{CLS_Estimators})--(\ref{CLS_Estimators_2}) of $\a$ and $\mu_\vareps$. The following assertions hold.

\begin{enumerate}

\item[(i)] $\widehat\a_{n,cls}$ and $\widehat{(\mu_\vareps)}_{n,cls}$ are strongly consistent, that is,  ${\widehat\a}_{n,cls} \longrightarrow \a$ a.s. and \\ $\widehat{(\mu_\vareps)}_{n,cls} \longrightarrow \mu_\vareps$ a.s  as $n\to\infty$.

\item[(ii)] $\Bigl(\widehat\a_{n,cls},\ \widehat{(\mu_\vareps)}_{n,cls}\Bigr)'$ satisfy the central limit theorem (BVN is short for bivariate normal),
\begin{equation}\label{CLT_mu_vareps}
\sqrt{n}\begin{bmatrix}
\widehat\a_{n,cls}-\a\\
\widehat{(\mu_\vareps)}_{n,cls}-\mu_\vareps
\end{bmatrix}
\buildrel d \over \longrightarrow
\text{BVN}\Biggl(
\begin{bmatrix} 0\\0\\ \end{bmatrix},\Sig\Biggr)
\qquad \text{as } n\to \infty,
\end{equation}
where 
$\Sig=
\begin{bmatrix}
\Sig_{11} & \Sig_{12}\\
\Sig_{21} & \Sig_{22}
\end{bmatrix}
$
and, referring to the notation above,

$\Sig_{11}=\bigl(\sig_G^2\mu_3^{(X)}+c^2\sig^2\bigr)\sig^{-4}$, %
$\Sig_{12}=\Sig_{21}=-\bigl(\mu\sig_G^2\mu_3^{(X)}+\mu c^2\sig^2-\sig_G^2\sig^4\bigr)\sig^{-4}$ , and
$\Sig_{22}=\bigl(\mu^2\sig_G^2\mu_3^{(X)}+\mu^2 c^2 \sig^2+\sig_\vareps^2\sig^4-\mu\sig_G^2\sig^4\bigl)
\sig^{-4}.$
\end{enumerate}
\end{proposition}

Relying on the formula $\mu=\mu_\vareps/\ova$, we define an estimator for $\mu$
\begin{equation}\label{mu_CLS}
\widehat\mu_{n,cls}=\frac{\widehat{(\mu_\vareps)}_{n,cls}}{1-\widehat\a_{n,cls}}
\end{equation}

We obtain the analogue of Proposition 5.1 for the pair $(\a,\mu)$. We rely principally on the delta method (see \cite{VDV}), which we briefly summarize in the bivariate case. Assume that in some setting the couple $(\widehat{\del}n,\widehat{\kappa}_n)$ consists of estimators of the parameters $(\del,\kappa)$ and that
\begin{equation}\label{CLT_delta}
\sqrt{n}\begin{bmatrix}
\widehat{\del}_n-\del\\
\widehat{\kappa}_n-\kappa
\end{bmatrix}
\buildrel d \over \longrightarrow
\text{BVN}\Biggl(
\begin{bmatrix} 0\\0\\ \end{bmatrix},W\Biggr)
\qquad \text{as } n\to \infty. 
\end{equation}
Let $g:\mathbb{R}^2\longrightarrow \mathbb{R}^2$, $g(x,y)=(g_1(x,y),g_2(x,y))$, be a twice continuously differentiable function with Jacobian $J_g(x,y)$. Then
\begin{equation}\label{CLT_delta2}
\sqrt{n}\begin{bmatrix}
g_1(\widehat{\del}n,\widehat{\kappa}_n))-g_1(\del,\kappa))\\
g_2(\widehat{\del}n,\widehat{\kappa}_n))-g_2(\del,\kappa)
\end{bmatrix}
\buildrel d \over \longrightarrow BVN\Biggl(\begin{bmatrix} 0\\0\\ \end{bmatrix},
J_g(\del,\kappa) W (J_g(\del,\kappa))'\Biggr)
\qquad \text{as } n\to \infty. 
\end{equation}.

\begin{proposition} The following assertions hold.

\begin{enumerate}
\item[(i)] $\widehat\mu_{n,cls}$ is a strongly consistent estimator of $\mu$.
\item[(ii)] $\Bigl(\widehat\a_{n,cls},\ \widehat\mu_{n,cls}\Bigr)'$ satisfies the central limit theorem, 
\begin{equation}\label{CLT_mu_X}
\sqrt{n}\begin{bmatrix}
\widehat\a_{n,cls}-\a\\
\widehat\mu_{n,cls}-\mu
\end{bmatrix}
\buildrel d \over \longrightarrow
\text{BVN}\Biggl(
\begin{bmatrix} 0\\0\\ \end{bmatrix},
J\,\Sig\, J'\Biggr)
\qquad \text{as } n\to \infty,
\end{equation}
where the matrix $\Sig$ is as in (\ref{CLT_mu_vareps}), Proposition 5.1-(ii), and
$ 
J=\ova^{-1}\begin{bmatrix}
\ova & 0\\
\mu & 1
\end{bmatrix}.
$
\end{enumerate}
\end{proposition}

\begin{proof} We apply the delta method to $(\a,\mu_\vareps)$ and their estimators $(\widehat\a_{n,cls},\widehat{(\mu_\vareps)}_{n,cls})$. Let the function $g:\mathbb{R}^2\longrightarrow \mathbb{R}^2$, $g(x,y)=\bigl(x,\frac{y}{1-x}\bigr)$, restricted to $x\in (0,1)$ and $y>0$. Its Jacobian is
\[
J_g(x,y)=
\begin{bmatrix}
1 & 0\\
\frac{y}{(1-x)^2} & \frac{1}{1-x}
\end{bmatrix}.
\]
Noting that $g(\a,\mu_\vareps)=(\a,\mu_\vareps/\ova)=(\a,\mu)$ and $J_g(\a,\mu_\vareps)=\ova^{-1}
\begin{bmatrix}
\ova & 0\\
\mu & 1
\end{bmatrix}.$
Proposition 5.1-(ii) and (\ref{CLT_mu_vareps}) imply that (\ref{CLT_delta}) holds for $(\a,\mu_\vareps)$, which in turn implies (\ref{CLT_mu_X}) via (\ref{CLT_delta2}). 
\end{proof} 

\subsection{The Yule-Walker estimation method for the means}

We discuss the Yule-Walker (YW) estimators for $\a$, $\mu$, and $\mu_\vareps$. We use the standard method of moments formulas for $\a$ and $\mu$, namely,
\begin{equation}\label{HW_alpha_mu}
\widehat{\a}_{n,yw}=\frac{\sum_{t=1}^{n-1}(X_t-\ov{X}_n)(X_{t+1}-\ov{X}_n)}{\sum_{t=1}^n (X_t-\ov{X}_n^2)}\quad \text{and} \quad 
\widehat{\mu}_{n,yw}=\ov{X}_n,
\end{equation}
where $\ov{X}_n=\frac{1}{n}\sum_{t=1}^n X_t$. For $\mu_\vareps$, we use
\begin{equation}\label{HW_mu_vareps}
\widehat{(\mu_\vareps)}_{n,yw}=(1-\widehat{\a}_{n,yw})\widehat{\mu}_{n,yw}.
\end{equation}

The next result states that the CLS and YW estimators of $\a$, $\mu_\vareps$, and $\mu$ are asymptotically close, implying in particular that for $n$ large enough, $\bigl(\widehat{\a}_{n,yw}-\a, \widehat{(\mu_\vareps)}_{n,yw}-\mu_\vareps\Bigr)'\sim \text{BVN}(\bigl((0,0)',\Sig\bigr)$ (see (\ref{CLT_mu_vareps})) and $\bigl(\widehat{\a}_{n,yw}, \widehat{\mu}_{n,yw}\Bigr)'\sim \text{BVN}(\bigl((0,0)',J\,\Sig\, J'\bigr)$ (see (\ref{CLT_mu_X})).

\begin{proposition}
Let $\{X_t\}$ be the novel negative binomial \inar1 process. Then the CLS and the YW estimators of $\a$, $\mu$ and $\mu_\vareps$ are asymptotically equivalent, that is, $\widehat{\a}_{n,yw}-\widehat{\a}_{n,cls}=o_p(n^{-1/2})$,  $\widehat{(\mu_\vareps)}_{n,yw}-\widehat{(\mu_\vareps)}_{n,cls}=o_p(n^{-1/2})$, and $\widehat{\mu}_{n,yw}-\widehat{\mu}_{n,cls}=o_p(n^{-1/2})$. 
\end{proposition}

\begin{proof} We proceed as in \cite{FRMC}. We divide the denominators of $\widehat{\a}_{n,yw}$ and $\widehat{\a}_{n,cls}$ by $n$ and denote them by $D_{n,yw}$ and $D_{n,cls}$, respectively. Note $\lim_{n\to \infty}D_{n,yw}=\lim_{n\to \infty}D_{n,cls}=A-\mu^2$, where $A=\lim_{n\to\infty}n^{-1}\sum_{t=1}^nX_{t-1}=\lim_{n\to\infty}n^{-1}\sum_{t=1}^{n-1}X_{t-1}$ is finite by ergodicity. Fairly lengthy but straighforward algebraic manipulations lead to
\begin{align*}
&\sqrt{n}\bigl(\widehat{\a}_{n,yw}-\widehat{\a}_{n,cls}\bigr)=\frac{X_0X_1D_{n,yw}(1/n
-1)}{D_{n,yw}D_{n,cls}\sqrt{n}}+\frac{D_{n,cls}-D_{n,yw}}{D_{n,yw}D_{n,cls}}\cdot\frac{\sum_{t=2}^n X_tX_{t-1}}{\sqrt{n}}\\ &-\frac{D_{n,cls}\ov{X}_n-D_{n,yw}(\ov{X}_n-X_1/n)}{D_{n,yw}D_{n,cls}}\cdot\frac{\sum_{t=2}^nX_{t-1}}{\sqrt{n}}-\frac{D_{n,cls} \ov
{X}_n}{D_{n,yw}D_{n,cls}}\frac{\cdot{\ov{X}_n-X_1}}{\sqrt{n}}.
\end{align*}
Therefore,
\[
\sqrt{n}\bigl(\widehat{\a}_{n,yw}-\widehat{\a}_{n,cls}\bigr)=o_p(1)+o_p(1)O_p(1)-o_p(1)O_p(1)-o_p(1)=o_p(1),
\]
which implies $\widehat{\a}_{n,yw}-\widehat{\a}_{n,cls}=o_p(n^{-1/2})$. For the case of $\mu_\vareps$, we have by
(\ref{CLS_estimators_2p}) and (\ref{HW_mu_vareps}),
\begin{align*}
&\sqrt{n}\Bigl(\widehat{(\mu_\vareps)}_{n,yw}-\widehat{(\mu_\vareps)}_{n,cls}\Bigr)=\frac{\sum_{t=1}^nX_t-\widehat{\a}_{n,cls}\sum_{t=1}^nX_{t-1}-n\ov{X}_n(1-\widehat{\a}_{n,cls})}{\sqrt{n}}=\\
& \ov{X}_n\sqrt{n}\bigl(\widehat{\a}_{n,yw}-\widehat{\a}_{n,cls}) +\widehat{\a}_{n,cls}\cdot\frac{X_n+X_0}{\sqrt{n}}=o_p(1)+o_p(1)=o_p(1). 
\end{align*}
Finally, we have by (\ref{HW_alpha_mu})and (\ref{HW_mu_vareps}),
\[
\sqrt{n}\Bigl(\widehat{\mu}_{n,yw}-\widehat{\mu}_{n,cls}\Bigr)=\frac{\sqrt{n}\Bigl(\widehat{(\mu_\vareps)}_{n,yw}-\widehat{(\mu_\vareps)}_{n,cls}\Bigr)}{1-\widehat{\a}_{n,cls}}=o_p(1).
\]
\end{proof}

\subsection{The conditional maximum likehood (CML) method for the means}

The conditional log-likelihood function given $X_0$ is
\[
L(\a,\mu,r)=\log\Big[\prod_{t=1}^n P(X_t|X_{t-1})\Big]=\sum_{t=1}^n\log\,P(X_t|X_{t-1}),
\]
where $P(X_t=j|X_{t-1}=i)$ is given in (\ref{trans_order_1_NB}). The CML estimators $(\widehat{\a}_{n,cml},\widehat{\mu}_{n,cml}, \widehat{r}_{n,cml})$ of $(\a,\mu, r)$ that maximize $L(\a,\mu,r)$, and as such, they must be solutions to the non-linear system of equations $\frac{\partial L}{\partial (\a,\mu,r)}=0$, where $\frac{\partial L}{\partial (\a,\mu,r)}=\Bigl(\frac{\partial L}{\partial \a},\frac{\partial L}{\partial \mu} ,\frac{\partial L}{\partial r}\Bigr)'$. In general, CML estimators cannot be obtained in closed form. The standard approach is to rely on a suitable numerical method of optimization (see \cite{GBSO} for example).  

\subsection{The conditional least squares method for the variances} 

Next, we derive estimators for the variances $\sig_G^2$, $\sig_\vareps^2$, and $\sig^2$. We will rely on results in $\cite{WIN}$. We define the conditional sum of squares for variances
\begin{equation}\label{S_CLS}
S_n(\a,\mu_\vareps,\sig_G^2,\sig_\vareps^2)=\sum_{t=1}^n \bigl(U_t^2-E(U_t^2|X_{t-1})\bigr)^2.
\end{equation}
where $U_t=X_t-E(X_t|X_{t-1})$. Since $E(X_t|X_{t-1})=\a X_{t-1}+\mu_\vareps$ and $E(U_t^2|X_{t-1})=Var(X_t|X_{t-1})=\sig_G^2X_{t-1}+\sig_\vareps^2$ (cf. Proposition 4.1-(ii)), we proceed as in \cite{WIN} and replace $\a$ and $\mu_\vareps$ by their CLS estimators, $\widehat\a_{n,cls}$ and $\widehat{(\mu_\vareps)}_{n,cls}$, of (\ref{CLS_Estimators}) and (\ref{CLS_Estimators_2}), respectively, and thus modify the function $S_n(\a,\mu_\vareps,\sig_G^2,\sig_\vareps^2)$ of (\ref{S_CLS}) to
\begin{equation}\label{S_CLS_2}
\widehat{S}_n(\sig_G^2,\sig_\vareps^2)=\sum_{t=1}^n \bigl(\widehat{U}_t^2-\sig_G^2X_{t-1}-\sig_\vareps^2\bigr)^2,
\end{equation}
where $\widehat{U}_t=X_t-\widehat\a_{n,cls} X_{t-1}-\widehat{(\mu_\vareps)}_{n,cls}$.
The estimators of $\sig_G^2$, $\sig_\vareps^2$ that minimize $\widehat{S}_n(\sig_G^2,\sig_\vareps^2)$ are
\begin{equation}\label{Sig_G_hat}
\widehat{(\sig_G^2)}_{n,cls}=\frac{n\sum_{t=1}^n \widehat{U}_t^2X_{t-1}-\sum_{t=1}^n X_{t-1}\sum_{t=1}^n \widehat{U}_t^2}{n\sum_{t=1}^n X_{t-1}^2-\bigl(\sum_{t=1}^n X_{t-1}\bigr)^2}
\end{equation}
and
\begin{equation}\label{Sig_VAREPS_hat}
\widehat{(\sig_\vareps^2)}_{n,cls}=\frac{\sum_{t=1}^n X_{t-1}^2\sum_{t=1}^n 
\widehat{U}_t^2-\sum_{t=1}^n \widehat{U}_t^2X_{t-1}\sum_{t=1}^n X_{t-1}}{n\sum_{t=1}^n X_{t-1}^2-\bigl(\sum_{t=1}^n X_{t-1}\bigr)^2},
\end{equation}
or more compactly,
\[
\widehat{(\sig_\vareps^2)}_{n,cls}=\frac{1}{n}\Biggl[\sum_{t=1}^n \widehat{U}_t^2-\widehat{(\sig_G^2)}_{n,cls}\sum_{t=1}^n X_{t-1}\Biggr].
\]
Let $X$ be a generic random variable with an $NB(r,\mu)$ distribution (the marginal distribution of $\{X_t\}$). We define or recall a a useful function of $X$ along with two key matrices (cf. \cite{WIN}). Let 
\begin{equation}\label{R_X}
R(X)=2\sig_G^4X^2+(\mu_4^{(G)}+4\sig_G^2\sig_\vareps^2-3\sig_G^4)X+\mu_4^{(G)}-\mu_4^{(\vareps)};
\end{equation}
\begin{equation}\label{Sig_X}
\Sig_1=\begin{bmatrix}
E[R(X)X^2] & E[R(X)X]\\
E[R(X)X] & E[R(X)]
\end{bmatrix};
\end{equation}
\begin{equation}\label{Phi_X}
\Phi=\begin{bmatrix}
E(X^2) & E(X)\\
E(X) & 1
\end{bmatrix}
\quad \text{and}\quad 
\Phi'^{-1}=\sig^{-2}\begin{bmatrix}
1 & -E(X)\\
-E(X) & E(X^2)
\end{bmatrix}.
\end{equation}

The following result is a direct consequence of Theorem 3.3 in \cite{WIN}. The proof is omitted. As previously noted, the basic assumption for part (ii)  that $\mu_4^{(G)}$ and $\mu_4^{(\vareps)}$ be finite  holds. 

\begin{proposition} Let $\widehat{(\sig_G^2)}_{n,cls}$ and $\widehat{(\sig_\vareps^2)}_{n,cls}$ be the CLS estimators  (\ref{Sig_G_hat})--(\ref{Sig_VAREPS_hat}) of $\sig_G^2$ and $\sig_\vareps^2$. The following assertions hold.

\begin{enumerate}
\item[(i)] $\widehat{(\sig_G^2)}_{n,cls}$ and $\widehat{(\sig_\vareps^2)}_{n,cls}$ are strongly consistent, that is,  $\widehat{(\sig_G^2)}_{n,cls} \longrightarrow \sig_G^2$ a.s. and $\widehat{(\sig_\vareps^2)}_{n,cls} \longrightarrow \sig_\vareps^2$ a.s  as $n\to\infty$.
\item[(ii)] $\Bigl(\widehat{(\sig_G^2)}_{n,cls},\ \widehat{(\sig_\vareps^2)}_{n,cls}\Bigr)'$ satisfy the central limit theorem
\begin{equation}\label{CLT_sigmas}
\sqrt{n}\begin{bmatrix}
\widehat{(\sig_G^2)}_{n,cls}-\sig_G^2\\
\widehat{(\sig_\vareps^2)}_{n,cls}-\sig_\vareps^2
\end{bmatrix}
\buildrel d \over \longrightarrow
\text{BVN}\Biggl(\begin{bmatrix} 0\\0\\ \end{bmatrix}, \Phi^{-1}\Sig_1\ (\Phi^{-1})'
\Biggr),
\qquad \text{as } n\to \infty.
\end{equation}
with $\Sig$ and $\Phi$ as in (\ref{R_X})--(\ref{Phi_X})
\end{enumerate}
\end{proposition}

The formulas $\displaystyle\sig^2=\frac{\mu\sig_G^2+\sig_\vareps^2}{1-\a^2}$ (see (\ref{c_squared}) and $\displaystyle r=\frac{\mu_\vareps^2}{\sig_\vareps^2-\mu_\vareps}$ suggest the following estimators for $\mu$, $\sig^2$, and $r$, respectively: 

\begin{enumerate}
\item[(a)]$\widehat{(\sig^2)}_{n,cls}=\frac{\widehat{(\mu_\vareps)}_{n,cls}\widehat{(\sig_G^2)}_{n,cls}+(1-\widehat{\a}_{n,cls})\widehat{(\sig_\vareps^2)}_{n,cls}}{{(1-\widehat{\a}_{n,cls})^2}(1+\widehat{\a}_{n,cls})}$ for $\sig^2$;
\item[(b)]$\widehat{r}_{n,cls}=\frac{\widehat{(\mu_\vareps)}^2}{\widehat{(\sig_\vareps^2)}_{n,cls}-\widehat{(\mu_\vareps)}_{n,cls}}$ for $r$.
\end{enumerate}

We conclude from Proposition 5.1-(i) and Proposition 5.3-(i) that
 
\begin{corollary}
$\widehat{(\sig^2)}_{n,cls}$, and $\widehat{r}_{n,cls}$ are strongly consistent estimators for $\sig^2$ and $r$, respectively.
\end{corollary}

\begin{remark} 
We note that if $\a$ and $\mu_\vareps$ are assumed known, then the CLS estimators of $\sig_G^2$ and $\sig_\vareps^2$ are obtained by minimizing the function 
\begin{equation}\label{S_CLS_3}
\widetilde{S}_n(\sig_G^2,\sig_\vareps^2)=\sum_{t=1}^n \bigl(U_t^2-\sig_G^2X_{t-1}-\sig_\vareps^2\bigr)^2,
\end{equation}
where, we recall, $U_t=X_t-\a X_{t-1}-\mu_\vareps$. The solutions are obtained by simply replacing the variable $\widehat{U}_t$ by $U_t$ in (\ref{S_CLS})--(\ref{Sig_VAREPS_hat}). 
\[
\widetilde{(\sig_G^2)}_{n,cls}=\frac{n\sum_{t=1}^n U_t^2X_{t-1}-\sum_{t=1}^n X_{t-1}\sum_{t=1}^n U_t^2}{n\sum_{t=1}^n X_{t-1}^2-\bigl(\sum_{t=1}^n X_{t-1}\bigr)^2}
\]
and
\[
\widetilde{(\sig_\vareps^2)}_{n,cls}=\frac{\sum_{t=1}^n X_{t-1}^2\sum_{t=1}^n 
U_t^2-\sum_{t=1}^n U_t^2X_{t-1}\sum_{t=1}^n X_{t-1}}{n\sum_{t=1}^n X_{t-1}^2-\bigl(\sum_{t=1}^n X_{t-1}\bigr)^2}.
\]
Therefore, by Proposition 3.1 in \cite{WIN}, $\widetilde{(\sig_G^2)}_{n,cls}$ and $\widetilde{(\sig_\vareps^2)}_{n,cls}$ are consistent estimators of $\sig_G^2$ and $\sig_\vareps^2$ and $\bigl(\widetilde{(\sig_G^2)}_{n,cls}-\sig_G^2, \widetilde{(\sig_\vareps^2)}_{n,cls}-\sig_\vareps^2\bigr)'$ is asymptotically normal (we omit the details).

An analogous result to Proposition 5.2 can be obtained in this case for the pair $(\sig_G^2,\sig^2)$ by way of the delta method (see (\ref{CLT_delta})--(\ref{CLT_delta2})). The corresponding pair of estimators is 
$\bigl(\widetilde{(\sig_G^2)}_{n,cls},\widetilde{(\sig^2)}_{n,cls}\bigr)$, where 
$\displaystyle \widetilde{(\sig^2)}_{n,cls}=\frac{\mu\widetilde{(\sig_G^2)}_{n,cls}+\widetilde{(\sig_\vareps^2)}_{n,cls}}{(1-\widehat\a_{n,cls})^2}$ and the function \\ $g:\mathbb{R}^2\longrightarrow \mathbb{R}^2$ is $g(x,y)=\bigl(x,\frac{\mu x+y}{1-\a^2}\bigr)$, with $g(\sig_G^2,\sig_\vareps^2)=(\sig_G^2,\sig^2)$. 
\end{remark}

\end{document}